\documentclass[hidelinks,onefignum,onetabnum]{siamart250211}  	 
\usepackage{graphicx}
\usepackage[english]{babel}
\usepackage{amsmath}
\usepackage{color}
\usepackage{float}
\usepackage{caption}
\usepackage{subcaption}
\usepackage{hyperref}
\usepackage{lipsum}
\usepackage{amsfonts}
\usepackage{graphicx}
\usepackage{epstopdf}
\usepackage{algorithmic}
\ifpdf
  \DeclareGraphicsExtensions{.eps,.pdf,.png,.jpg}
\else
  \DeclareGraphicsExtensions{.eps}
\fi

\newsiamremark{remark}{Remark}

\headers{Bounds for internal solutions from boundary data}{V. Druskin, S. Moskow, and M. Zaslavsky}

\title{On optimality and bounds for internal solutions generated from boundary data-driven Gramians}

\author{V. Druskin
\thanks{Worcester Polytechnic Institute, Department of Mathematical Sciences,
Stratton Hall,
100 Institute Road, Worcester MA, 01609 and Southern Methodist University, Department of Mathematics, Clements Hall, 3100 Dyer st,
Dallas, TX 75205 (\email{vdruskin1@gmail.com}).}
\and S. Moskow
\thanks{Department of Mathematics, Drexel University, Korman Center, 3141 Chestnut Street, Philadelphia, PA 19104
(\email{slm84@drexel.edu}).}
\and M. Zaslavsky
\thanks{Southern Methodist University, Department of Mathematics, Clements Hall, 3100 Dyer st,
Dallas, TX 75205 (\email{mzaslavskiy@smu.edu}).}}

\usepackage{amsopn}

\begin{document}
\maketitle 
\begin{abstract} We consider the computation of internal solutions for a time domain plasma wave equation with unknown coefficients from the data obtained by sampling its transfer function at the boundary. The computation is performed by transforming known background snapshots using the Cholesky decomposition of the data-driven Gramian.  We show that this approximation is asymptotically close to  the projection of the true internal solution onto the subspace of background snapshots.  This allows us to derive  a generally applicable bound for the error in the approximation of internal fields from boundary data %only 
for a time domain plasma wave equation with an unknown potential $q$. For general $q\in L^\infty$, we prove convergence of these data generated internal fields in one dimension for two examples. The first is for piecewise constant initial data and sampling $\tau$ equal to the pulse width. The second is piecewise linear initial data and sampling at half the pulse width. We show that in both cases the data generated solutions converge in $L^2$ at order $\sqrt{\tau}$. We present numerical experiments validating the result and the sharpness of this convergence rate. \end{abstract}
%for general $q\in L^\infty$
\section{Introduction}
%Kravchenko2020SomeRD
%Boumenir2020TransmutationOA

There is a vast literature on the inverse scattering problem of determining an unknown potential $q$ from boundary data for the plasma wave equation \begin{equation}\label{waveeqintro}
u_{tt}  -\Delta u + q(x) u=   0, 
\end{equation} see, for example \cite{9357476,gilman2015mathematical,WaAnLi,CaCoMo, Burfeindt2,lagergren2021deep,10352835,cheney2021procedure,Virieux2016AnIT,Kepley2016GeneratingVI, Malcolm2007IdentificationOI}.
It is well known that knowledge of the interior solution $u$ makes the reconstruction problem for $q$ much easier.   Indeed, some methods to compute $q$ from boundary data construct an approximate $u$ as an intermediary step \cite{Kepley2016GeneratingVI}, and the ability to obtain internal data is a main advantage to coupled physics, or hybrid inversion methods \cite{Ba}. 

One class of inversion methods use reduced order models, or ROMs. 
Reduced order models (ROMs) in this context are finite dimensional versions of (\ref{waveeqintro}) that can be determined from the boundary data, see for example \cite{doi:10.1137/1.9781611974829.ch7}, and have led to efficient inversion algorithms \cite{borcea2011resistor,borcea2014model,druskin2016direct,druskin2018nonlinear,borcea2017untangling,borcea2019robust}. In the work \cite{BoDrMaMoZa}, it was first found that one could use ROMs to produce accurate internal data from the boundary data only.  These approximate fields can then be combined with the Lippmann-Schwinger integral equation for image reconstruction \cite{DrMoZa, DrMoZa2, BoGaMaZi, DrMoZa3, DrMoZa4}, where the image quality depends on the accuracy of the approximation of the internal fields. The accuracy of these fields has mostly been explained in the following intuitive (but non rigorous) way: after sequential orthogonalization of the time snapshots, the perturbed and background orthogonalized snapshots are quite similar due to the cancellation of the reflected wave. 

One result about the data generated snapshots is in the appendix of \cite{BoGaMaZi}, where the authors construct an exact solution for the case of a piecewise constant medium in one dimension with interfaces exactly coinciding with the time steps.  They show that if the initial wave starts with support away from the interfaces, the data-generated solutions are exact. This is because for this special medium, the true snapshots are exactly a linear combination of the background snapshots. In this paper,  we consider a general setup, where obtaining the exact solution is not possible.  We 
show that the error in the data-generated solutions can be controlled by the error in the best approximation of the true solutions from the space of background snapshots, and these errors become asymptotically close if the error in the best approximation is small.   Then, for well-chosen initial waves, we obtain convergence of the data generated internal solutions for general one-dimensional media. The general error bounds hold for any SISO (single input-single output) problem regardless of dimension, however, in higher dimensions a SISO background subspace will not generally be rich enough for convergence. In higher dimensions a MIMO setup (multiple input -multiple output) will be necessary for convergence, and the approach presented in this manuscript provides a framework for a rigorous higher dimensional study. 
Convergence is based on the following steps:  
\begin{enumerate}
\item The data generated internal snapshots can be characterized as the unique (admissible) linear combination of background snapshots that exactly interpolate the true Gramian (mass matrix) $M$. 

\item For a sufficiently rich space of background snapshots, the best approximations will have a mass matrix that is close to the true mass matrix.  We show a stability result for the mapping from the mass matrix to the snapshots, and this relates the data-generated solutions to these best approximations. 

\item Since the background snapshots are in the approximation space, the best approximation error is based on the regularity of the scattered field, which is much more regular than the singular pulse.  
\end{enumerate}
The analysis shows that the shape of the pulse and the time sampling interval should be such that the background snapshots form a good approximation space; but not oversampled to cause the mass matrix to be badly conditioned. We prove convergence for two one-dimensional examples. In the first example we use a piecewise constant approximate delta pulse with sampling equal to the pulse width $\tau$.  In the second example, we probe with a piecewise linear approximate delta pulse with the time sampling $\tau$ at half of the pulse width. In both cases we have convergence of the solutions on the order of $\sqrt{\tau}$ for $q\in L^\infty$.

The paper is organized as follows. In Section 2, we describe precisely the inverse problem. The method to construct the approximate internal fields from the boundary data is 
described in Section 3, along with a proof that they interpolate the Gramian exactly. In Section 4 we prove a general error bound. The main results are Proposition 2 and Corollary 1, showing that these data generated internal fields have error that is asymptotically close to the error in the projections of the true solutions onto the background snapshot space. In Section 5 we show that by using a step function initial wave, we obtain convergence as the sampling interval approaches zero. We consider a piecewise linear hat initial wave in Section 6. In Section 7 we discuss  how this framework works in higher dimensions, and in Section 8 we show numerical experiments. Section 9 contains a discussion of other possible generalizations and future work. 
\section{Problem setup}  We consider first a single input/single output (SISO) problem. That is, for now we assume we have just one source collocated with the receiver, and
consider the following wave model problem for a domain $\Omega\subset\mathbb{R}^n$
\begin{equation}\label{waveeq}
u_{tt}  -\Delta u + q(x) u=   0 \ \ \mbox{in}  \ \Omega\times [ 0,\infty)
\end{equation}
with initial conditions
\begin{eqnarray} \label{waveeqic}
u ( t=0) &=& g \ \mbox{in}\  \Omega  \\
u_t ( t=0) &=&  0\  \mbox{in}\  \Omega\\ {\partial u\over{\partial\nu}} &=& 0 \ \mbox{on} \ \partial\Omega \label{waveeqid}\end{eqnarray}
where   $ g $ is initial data representing a localized source near the boundary,  and we assume homogeneous Neumann boundary conditions on the spatial boundary $\partial \Omega$. We assume $q(x)\ge 0$ is our unknown potential, not necessarily small, but with compact support.   The exact forward solution to (\ref{waveeq}-\ref{waveeqid}) is 
\begin{equation}\label{exactsolution}  u(x,t) = \cos{(\sqrt{-\Delta+q}t)} g(x), \end{equation}
where the square root and cosine are defined via the spectral theorem. This solution is assumed to be unknown, except near the receiver.  Assume we measure the solution back at the source at the $2n-1$ evenly spaced time steps $t= k\tau$ for $k=0,\dots, 2n-2$, the receiver modeled by 
\begin{eqnarray}\label{datadef} F(k\tau ) &= &\int_\Omega g(x) u(x,k\tau) dx\nonumber \\ &=&   \int_\Omega g(x) \cos{(\sqrt{-\Delta+q }k\tau)} g(x) dx .\end{eqnarray}
The inverse coefficient problem one may consider is as follows:  Given $$ \{ F(k \tau) \} \  \mbox{ for } \  k=0,\dots, 2n-2,$$
reconstruct $q$. To this end, we first compute, from the boundary data only, approximations of the internal snapshots $u_k = u(x,k\tau)$ for $k=0,\ldots, n-1$. It is the accuracy of these data generated internal solutions that is the subject of this paper. That is, the problem we consider in this manuscript is:
Given $$ \{ F(k \tau) \} \  \mbox{ for } \  k=0,\dots, 2n-2,$$
reconstruct  $$ \{ u(x,k\tau) \} $$  for $k=0, \ldots, n-1$.

\section{Construction of internal solutions and a characterization}
Consider the true snapshots $$u_k = u(x,k\tau).$$
An essential component of the reduced order model is the Gramian, or mass matrix $M$ given by 
\begin{equation} \label{massmatrixj} M_{kl}= \int_\Omega u_k u_l  dx  \end{equation}
for $k,l =0,\ldots,n-1$. It is well known that $M$ can be obtained from the data. To see this, $M$ can be written as 
\begin{equation} \label{massmatrixj2} M_{kl} = \int_\Omega  g(x) \cos{(\sqrt{-\Delta+q}k\tau)} \cos{(\sqrt{-\Delta +q }l\tau)} g(x) dx.  \end{equation}
thanks to the formula (\ref{exactsolution}).  Then, from (\ref{datadef}), (\ref{massmatrixj2}),  and the cosine angle sum formula, one has 
\begin{equation} \label{massfromdataj} M_{kl} = \frac{1}{2} \left( F((k-l)\tau) + F((k+l)\tau) \right),\end{equation}  a direct formula from the data. 
Now, let  $$U = [ u_0, \ldots, u_{n-1} ] $$ be a row vector of the true snapshots,  so we can write \begin{equation}\label{truemass} M = \int_\Omega U^\top U. \end{equation} Consider also the background field $u^0(x,t)$, the solution to (\ref{waveeq}-\ref{waveeqid}) with $q(x)=0$, which we can assume that we know. 
Let $$U_0 = [ u_0^0, \ldots, u_{n-1}^0 ] $$
be a row vector of the background snapshots $u_k^0 = u^0(x,k\tau)$, and let $$ M_0 = \int_\Omega U_0^\top U_0 $$ be the background mass matrix. 
To compute the data generated snapshots, we first compute the unique Cholesky decompositions 
$$ M = LL^\top \ \ \ \ \ \ M_0 = L_0 L_0^\top. $$
These Cholesky decompositions were first used to solve the coefficient inverse problem in \cite{DrMaThZa}.
From these we compute the row vector of data generated internal fields 
\begin{equation} \label{datageneratedfields}  {\bf U}  := U_0 L_0^{-\top} L^\top = U_0 T \end{equation}
where $T$ is upper triangular and
\begin{equation} \label{Teq} T= L_0^{-\top} L^\top. \end{equation}
Note that the entries (columns) of  ${\bf U}$ are in the finite dimensional space $\mathcal{U}_0$ generated by the entries (columns) of $U_0$, and  
${\bf U}$ is obtainable from the data. 
The reason that the approximation $ {\bf U} $ is consistent with the true internal fields is the following characterization. 
\begin{lemma} \label{massfit}The row vector of data generated internal fields  ${\bf U} = U_0 T $ for $T$ given by (\ref{Teq})  satisfies 
\begin{equation} \label{datamass} \int_\Omega {\bf U}^\top {\bf U} = M\end{equation} 
where $M$ is the true mass matrix (\ref{truemass}).  Furthermore, $T$ is the unique upper triangular transformation of $U_0$ with positive diagonal entries which has this property. 
\end{lemma}
\begin{proof}  The formula (\ref{datamass}) follows from direct calculation. For uniqueness, if $\tilde{T}$ is upper triangular with positive diagonal entries and  $$ \int_\Omega \tilde{T}^\top U_0^\top U_0 \tilde{T} = M,$$ then $$ \tilde{T}^\top M_0 \tilde{T} = M, $$ which means that $\tilde{T}^\top L_0 $ is a Cholesky factor of $M$, hence $\tilde{T}= T$ by uniqueness of the Cholesky factorization. \end{proof}
\begin{remark} 
We can relate ${\bf U}$ given by (\ref{datageneratedfields}) to the continuum formally as follows. If $u(x,t)$ is the internal solution for unknown $q(x)$ and $u_0(x,t)= \delta(x-t)$ is the background solution corresponding to $q_0(x)=0$ and singular pulse,  then for $x\ne t$ we have \begin{equation}\label{eq:transmute} {u(x,t)=\int_0^t {u}(t',t) u_0(x,t') d t'}.\end{equation} So, for an approximate delta initial wave, one has approximately that  \begin{equation}\label{eq:transmute2} {u(x,t)=\int_0^t \mathcal{T}(t',t) u_0(x,t') d t'}\end{equation} where kernel $\mathcal{T}(t',t)$, is close to being independent of $x$.  Its discretization in $t$ and $t'$ yields a decomposition of the snapshots of $u$ into a basis of background snapshots, and $\mathcal{T}(t',t)$ becomes a matrix, which should be triangular since it inherits the structure of $u(t',t)$.  For some special cases of piecewise constant layered media, the discrete counterpart of \eqref{eq:transmute2}  becomes exact \cite{bube1983one, BoGaMaZi}, however, its convergence in general has only been understood on an intuitive level.  This is related to the theory of transmutation operators and the Gelfand-Levitan-Marchenko theory   \cite{Kravchenko2020SomeRD,Agranovich1963TheIP,habashy1991generalized}, but the precise connections need to be explored further.  We note that the results in Section \ref{sec:generalbounds} hold for any initial wave, while the convergence in the examples in Sections \ref{sec:pwconstant}  and \ref{sec:pwlinear} requires that the initial pulse approaches a delta.   %\footnote{\color{red} Shari special: Let $u_0=g(x-t)$, for simplicity of notations assign $\mathcal{T}(t,t')=0$ for $t'>t$  and compute  $u'(x,t)=\int u(x,t-t'')g(t'')dt''$, that by changing order of integration would just replace $u_0$ with $g$, i.e., $u'(x,t)$ becomes linear superposition of $g(x-t')$. Does it make sense?}
%For localized pulses, (\ref{eq:transmute}) holds approximately, and
%This gives a discrete inverse scattering theory \cite{Natterer}, which is also related to the decomposition of one set of orthogonal polynomial with respect to another \cite{Boumenir2020TransmutationOA}. 
%In this work we show that this ROM based procedure yields a data generated discretization of (\ref{eq:transmute}) which is close to (or equal to) the  projection onto the space of background snapshots. 
%The mass matrix $M$ is an exact evaluation of its continuum version $$\mathcal{M}(s,t) = \int_\Omega u(x,s)u(x,t)dx$$
%on the time grid. Using (\ref{eq:transmute}) , \begin{eqnarray}\mathcal{M}(s,t) &=&\int_\Omega \int_0^s \mathcal{T}(s,s') u_0(x,s') d s'\int_0^t \mathcal{T}(t,t') u_0(x,t') d t' dx\\  &=& \int_0^s \int_0^t \mathcal{T}(s,s')\mathcal{M}_0(s',t')\mathcal{T}(t,t')  d t' d s' \label{masskernels}\end{eqnarray}
%where $\mathcal{M}_0$ is the background continuum mass kernel. Our choice of $T$ satisfies $$ M = T^\top M_0 T,$$
%a discrete version of (\ref{masskernels}), and retains the upper triangular (causal) structure of $\mathcal{T}$. Precise connections with the Marchenko-Gelfand-Levitan theory \cite{habashy1991generalized} need to be explored further. 
\end{remark}
Our data generated internal fields are of the form  ${\bf U} = {U}_0 T$ where $T$ is upper triangular, so the data generated internal field can only be composed of background snapshots from previous and current time steps. Our strategy to analyze the error ${\bf U}-U$ is to compare ${\bf U}$ to the best possible approximation that is of this form. We formalize this with a definition. 
\begin{definition} \label{causal} Using the sequential basis $U_0 = [ u_0^0, \ldots, u_{n-1}^0 ] $ of background snapshots for the space $\mathcal{U}_0$, we define the {\bf admissible set} $\mathcal{V}_0^n\subset [\mathcal{U}_0]^n$ to be $$\mathcal{V}_0^n =\{  V= U_0 \tilde{T}\ |\  \tilde{T}\ \mbox{upper triangular with positive diagonal entries}\}. $$  That is, for $V= [ v_0, \ldots v_{n-1} ]$, each $v_i$ is in the subspace generated by $u_0^0, \ldots, u_i^0$, and its coefficient of $u_i^0$ is positive.
\end{definition} 
Lemma \ref{massfit} shows that if the true snapshots are in $\mathcal{V}^n_0$, ${\bf U}$ will be exact. It was already shown in \cite{BoGaMaZi} that this is the case if the medium is piecewise constant with jumps exactly coinciding with the time steps in one dimension. In this paper we extend that work by showing more generally that the error in ${\bf U}$ can be controlled by the error in the best approximation of $U$ from $\mathcal{V}^n_0$.
That is, the data generated approximations to the internal fields ${\bf U}$ can be viewed as a nonlinear (oblique) projection of the true snapshots onto the admissible set $\mathcal{V}^n_0$,

\section{A general error bound} \label{sec:generalbounds}
Suppose that some approximation $$ \hat{U} = U_0 \hat{T} $$ is in the admissible set $\mathcal{V}_0^n$ from Definition \ref{causal}. Define $\hat{M}$ to be the mass matrix of these approximations,
$$\hat{M} = \int_\Omega \hat{U}^\top\hat{U},$$ and define $\hat{L}$ to be its unique Cholesky factor.  In order to measure the error, it will be convenient to use a particular norm on $[L^2(\Omega)]^n$.  For a set of functions $V=[ v_0, \ldots, v_{n-1} ]\in [L^2(\Omega)]^n$, we define 
\begin{equation} \| V\|_2 =\left(\sum_{i=0}^{n-1} \| v_i\|_2^2\right)^{1/2} .\label{L2n} \end{equation} With this norm, we see that if $V$ were time snapshots of a continuous  $v(x,t)$ with steps $\tau$, then
$$ {\| V\|_2\over{\sqrt{n}}} \approx\| v(x,t) \|_{L^2(\Omega\times[0,n\tau])}$$
using a left hand rule in time. The following Lemma states that the error between the data generated snapshots and any other admissible approximation is the norm difference of their corresponding Cholesky factors. 

\begin{lemma} \label{froblem} Let  ${\bf U} = U_0 T $ for $T$ given by (\ref{Teq}) be the row vector of data generated internal fields. For any admissible approximation $\hat{U} = U_0 \hat{T} \in \mathcal{V}_0^n$, let $\hat{M}$ be its mass matrix and $\hat{L}$ its Cholesky factor. Then we have
\begin{equation} {\| {\bf U}- \hat{U}\|_2} =  {\| L-\hat{L}\|_F},  \end{equation} and furthermore
\begin{equation} {\| {\bf U}\|_2 = \| { U} \|_2 } =  {\| L\|_F},  \end{equation}
where $U$ is the vector of true snapshots, $L$ is the Cholesky factor of the true mass matrix $M$, and $\| \cdot \|_F$ refers to the matrix Frobenius norm.
\end{lemma}
\begin{proof}
We can calculate 
\begin{eqnarray} \int_\Omega ( {\bf U}- \hat{U})^\top ({\bf U}- \hat{U}) &=& \int_\Omega ( T-\hat{T})^\top U_0^\top U_0 (T-\hat{T}) \nonumber \\ 
&=& ( T-\hat{T})^\top M_0 (T-\hat{T}). \end{eqnarray}
Recalling that $T= L_0^{-\top}L^\top $, we note also that we must have  $\hat{T} = L_0^{-\top} \hat{L}^\top$, and this yeilds  
\begin{equation} \int_\Omega ( {\bf U}- \hat{U})^\top ({\bf U}- \hat{U}) = (L-\hat{L})(L-\hat{L})^\top. \label{errormassform}\end{equation}
From taking the trace of both sides of (\ref{errormassform}) we get that 
\begin{eqnarray} \mbox{trace}\left( \int_\Omega ( {\bf U}- \hat{U})^\top ({\bf U}- \hat{U}) \right) &=& \mbox{trace}\left( (L-\hat{L})(L-\hat{L})^\top \right) \nonumber \\ &=& \| L-\hat{L}\|^2_F \nonumber \end{eqnarray} 
by the definition of Frobenius norm. 
%& \leq & C\| L\|^2  \kappa ( M )^2 \| M-\hat{M} \|^2 .\end{eqnarray}
The diagonal of the above matrix contains the squares of the $L^2$ function errors, so the above says that 
\begin{equation} \sum_{i=0}^{n-1} \| {\bf u}_i -\hat{u}_i \|_2^2 = \| L-\hat{L}\|_F^2. \end{equation}
Also, using Lemma \ref{massfit}, we see that \begin{equation} \sum_{i=0}^{n-1} \| {\bf u}_i\|_2^2 =\sum_{i=0}^{n-1} \| u_i\|_2^2 = \mbox{trace}( M ) =\| L\|^2_F.\end{equation}\end{proof}
We will need the following Theorem on the forward stability of Cholesky factorization, see for example \cite{ChPaSt}. 
\begin{theorem} \label{Llem}(Stewart, Sun) Let $L$ be the Cholesky factor of $M$ and $\hat{L}$ be the Cholesky factor of $\hat{M}$. Let  $\kappa_2( M )$ be  the condition number of $M$, $$\kappa_2(M)= \| M\|_2\cdot \| M^{-1}\|_2 .$$  Let 
\begin{equation} \epsilon = \| M - \hat{M}\|_F /\| M\|_2 .\end{equation}  Then if $\epsilon \kappa_2(M) < 1$, we have 
$$ \| L-\hat{L} \|_F \leq {1\over{\sqrt{2}}}\| L\|_2\kappa_2( M ) \epsilon +O(\epsilon^2).$$
\end{theorem}
We can combine this Theorem with Lemma \ref{froblem} to relate the difference between the data generated internal solutions and any admissible approximation to $U$ to the difference in their mass matrices. This also shows general stability for the data generated fields with respect to purturbations in the mass matrix. 
\begin{proposition} \label{firstprop} Let $U= [u_0, \ldots, u_{n-1} ] $ be the vector of true internal snapshots, $M$ be the true mass matrix (\ref{massmatrixj}), and let  $${\bf U} = U_0 T = [ {\bf u}_0, \ldots, {\bf u}_{n-1} ]$$ for $T$ given by (\ref{Teq}) be the row vector of data generated internal fields. If  $\hat{U}= [ \hat{u}_0,\ldots,\hat{u}_{n-1}]$  is any admissible approximation from $[\mathcal{U}_0]^n$ to the true snapshots $U$, and $\hat{M}$ is their corresponding mass matrix, then for \begin{equation} \epsilon = \kappa_2(M){ \| M-\hat{M}\|_F \over{ \| M\|_2}} \label{epseq}\end{equation} small enough, 
\begin{equation} \| {\bf U}- \hat{U} \|_2 \leq \epsilon \| U \|_2  . \end{equation} 
\end{proposition}
\begin{proof}
 Combining Lemma \ref{froblem} with Theorem \ref{Llem}, we have 
  \begin{equation}   \| {\bf U}- \hat{U}\|_2\
\leq \kappa_2(M)\| L\|_2   { \| M - \hat{M}\|_F \over {\| M\|_2 }},\end{equation}
if we assume that   $\epsilon <1$ is small enough. Since $\| L\|_2\leq  \| L\|_F  =\| U\|_2$ , the result follows.  \end{proof}
Observe that the admissible space $\mathcal{V}_0^n$ is a subset of the subspace of $[\mathcal{U}_0]^n$ containing upper triangular transformations of $U_0$. So, if the orthogonal projection of the true snapshots onto this subspace has positive diagonal entries (which physically we expect), it will be admissible, and this will be the best admissible transformation.  In this next Lemma, we find a bound for the error in the mass matrix for this projection. 
\begin{lemma} \label{masslem}
Let $U= [u_0, \ldots, u_{n-1} ] $ be the vector of true internal snapshots, $M$ be the true mass matrix (\ref{massmatrixj}). Define $\hat{U}= [ \hat{u}_0,\ldots,\hat{u}_{n-1}]$  to be the $L^2$ projections of the true snapshot where each $u_i$  is projected onto $\mbox{span}\{ u^0_0, \ldots ,u^0_i\}$, and let $\hat{M}$ be their mass matrix. Then
\begin{equation} \| M -\hat{M}\|_F \leq \| U - \hat{U}\|^2_2  + \| R\|_F\end{equation}
where $$ R_{ij} = \begin{cases}\int_\Omega (u_i-\hat{u}_i)\hat{u}_j & i <j \\ \int_\Omega \hat{u}_i(u_j-\hat{u}_j) & i > j \\ 0 & i=j \end{cases}$$
\end{lemma}
\begin{proof} Consider 
\begin{eqnarray} (M-\hat{M})_{ij} &=& \int_\Omega u_iu_j - \int_\Omega \hat{u}_i\hat{u}_j\nonumber \\
&=& \int_\Omega u_i(u_j-\hat{u}_j) + \int_\Omega (u_i-\hat{u}_i)\hat{u}_j,\nonumber\\
&=& \int_\Omega (u_i-\hat{u}_i)(u_j-\hat{u}_j)  + R_{ij} \nonumber \end{eqnarray}
where we have used that the $u_i-\hat{u}_i$  are orthogonal to $\hat{u}_j$ for  $j\leq i$.
So, 
$$ | (M-\hat{M})_{ij} | \leq \| u_i- \hat{u}_i\|_2 \| u_j -\hat{u}_j\|_2 + | R_{ij}|$$
where the first terms on the right hand side are the entries of an outer product and therefore 
\begin{eqnarray} \| M-\hat{M} \|_F &\leq&  \|  v^\top v \|_F  +\| R\|_F \nonumber \\
&\leq&  \|  v\|^2   +\| R\|_F\nonumber\end{eqnarray}
where the norm $\| \cdot\|$ denotes the vector $2$ norm and
 $$v = [ \| u_0-\hat{u}_0\|_2,\ldots, \| u_{n-1}-\hat{u}_{n-1}\|_2 ]$$  
 which proves the desired result.
 \end{proof}
The previous results combine to give us a general bound for the error in the data generated internal solutions. 
\begin{proposition} \label{mainprop}
Let $U= [u_0, \ldots, u_{n-1} ] $ be the vector of true internal snapshots, $M$ be the true mass matrix (\ref{massmatrixj}), let  $${\bf U} = U_0 T = [ {\bf u}_0, \ldots, {\bf u}_{n-1} ]$$ for $T$ given by (\ref{Teq}) be the row vector of data generated internal fields, and let  $\hat{U}= [ \hat{u}_0,\ldots,\hat{u}_{n-1}]$  to be the $L^2$ projections of the true snapshot where each $u_i$  is projected onto $\mbox{span}\{ u^0_0, \ldots ,u^0_i\}$, that is, $\hat{U}=U_0\hat{T}$ for $\hat{T}$ upper triangular. If $\hat{T}$ has positive diagonal entries, then for $\epsilon$ given by (\ref{epseq}) small enough 
%\begin{equation} \epsilon \leq {\kappa(M) \sqrt{n}\over{\| U\|_2^2}}\left(\| U - \hat{U}\|^2_2 + \|R\|_F\right). \end{equation}
\begin{equation} { \| {\bf U}- \hat{U} \|_2 \over{\sqrt{n}}} \leq{\kappa(M) \over{\| U\|_2}}\left(\| U - \hat{U}\|^2_2 + \|R\|_F\right) . \end{equation} 
and hence by the triangle inequality
\begin{equation} { \| {\bf U}- U \|_2 \over{\sqrt{n}}} \leq{\kappa(M) \over{\| U\|_2}}\left(\| U - \hat{U}\|^2_2 + \|R\|_F\right) + {\| U-\hat{U}\|_2\over{\sqrt{n}}}  \end{equation} 
where $R$ is given as in Lemma \ref{masslem}.
\end{proposition}
\begin{proof} Since $\hat{T}$ is upper triangular by definition, if it has positive diagonal entries then $\hat{U}$ is admissible, and we can apply Proposition \ref{firstprop} in addition to Lemma \ref{masslem}. We also use Lemma \ref{froblem} to get  $\| U\|_2^2= \| L\|^2_F=\| M \|_F \leq \sqrt{n} \| M\|_2 $ and obtain the first inequality. 
\end{proof}
Consider now the best $L^2$ approximation to $U$ from the full space $[\mathcal{U}_0]^n$, the orthogonal projections, which we will call $P_\tau U$. We may have that $P_\tau U$ is admissible. In this case, $P_\tau U ={\hat U}$, and the term $R$ in the above theorem is zero due to the full orthogonality. 
\begin{corollary}\label{projcor}
    Let $U= [u_0, \ldots, u_{n-1} ] $ be the vector of true internal snapshots, $M$ be the true mass matrix (\ref{massmatrixj}), let  $${\bf U} = U_0 T = [ {\bf u}_0, \ldots, {\bf u}_{n-1} ]$$ for $T$ given by (\ref{Teq}) be the row vector of data generated internal fields, and let  $P_\tau{U}$ be the $L^2$ projections of the true snapshot onto the full space $[\mathcal{U}_0]^n$. If $P_\tau U\in \mathcal{V}_0^n$ (i.e. it is in the admissible set), then for $\epsilon$ given by (\ref{epseq}) small enough 
%\begin{equation} \epsilon \leq {\kappa(M) \sqrt{n}\over{\| U\|_2^2}}\left(\| U - \hat{U}\|^2_2 \right). \end{equation}
\begin{equation} { \| {\bf U}- P_\tau{U} \|_2 \over{\sqrt{n}}} \leq{\kappa(M) \over{\| U\|_2}}\| U - P_\tau{U}\|^2_2 . \end{equation} 
\end{corollary}
\begin{remark}
Corollary \ref{projcor}  says that if the full projection is in the admissible space and has small error, the data generated solutions are asymptotically close this projection. In general, the full projection will not be in $\mathcal{V}_0^n$, and the matrix $R$ accounts for it not conforming to our admissible space.  For localized initial data,  $R$ will have a banded structure with small entries thanks to causality.  
An important point is that the background snapshots are exactly in the approximation space and admissible set $\mathcal{V}_0^n
\subset[\mathcal{U}_0]^n$, so for any projection, the error will depend on the regularity of  $U-U_0$. That is, if the initial wave is an approximate delta, $U$ itself is not regular, but the most singular part of $U$ is approximated exactly. 
\end{remark}
%I think 
%$$  \| M-\hat{M}\| \leq C \| U-\hat{U} \|_2 \| U \|_2 $$
%but need to check. Hopefully this means that we can show:
%\begin{equation} \| {\bf U}- U \|_2 \leq C\left( \| L\|  \kappa(M)\| U\|_2 + 1\right) \| U-\hat{U} \|_2. \end{equation} 
%The devil is in the details now, since the condition number of $M$ may blow up , and our $U$ are likely not bounded in $H^2$ as we take our pulse width and/or time steps to zero. Also haven't been careful with the matrix norms

\section{An example with a step function initial wave}\label{sec:pwconstant}
Let us consider the one dimensional case of (\ref{waveeq}-\ref{waveeqid}) where $\Omega = (0,L)\subset \mathbb{R} $ for $L$ large, 
\begin{equation}\label{waveeq1d}
u_{tt}  - u_{xx} + q(x) u=   0 \ \ \mbox{in} \  (0,L) \times [ 0,\infty)
\end{equation}
with initial and boundary conditions
\begin{eqnarray} \label{waveeqic1d}
u ( t=0) &=& g \ \mbox{in}\ (0,L)  \\
u_t ( t=0) &=&  0\  \mbox{in}\  (0,L) \\  u_x(x=0) =u_x(x=L) &=& 0 \ \mbox{for all } \  t\in [0,\infty).\label{waveeqid1d}\end{eqnarray}
We measure the solution back at the source at the $2n-1$ evenly spaced time steps $t= k\tau$ for $k=0,\dots, 2n-2$, 
\begin{equation}\label{datadef1d} F(k\tau ) = \int_0^1 g(x) u(x,k\tau) dx .\end{equation}
We will take initial data $g$ to be a piecewise constant function source near $x=0$. Define the reference step function 
\begin{equation} H(x) = \left\{ \begin{array}{cc} 1 & -1/2\leq x \leq 1/2 \\  0 & \mbox{otherwise} \end{array} \right. , \end{equation}
 and for a given time sampling size $\tau$ let us take initial data \begin{equation} g(x) =2{H({x/\tau})\over{\tau}}.\label{initaldata} \end{equation}
Note that the scaling is chosen so that $g$ is an approximate $\delta$ pulse function, and for any $\tau$, $$\int_0^{1} g(x)dx =1.$$ The background solution at the $k$th step for $k\geq 1$ will be  $$ u^0_k = u^0(x,k\tau ) = g(x- k\tau)/2 = {H((x-k\tau)/\tau)\over{\tau}},$$
and so the background snapshots are simply piecewise constant functions on a grid with nodes $(k+1/2)\tau$ for $k=0,\ldots,n-1$, with no overlapping support. Defining $T= (n-1/2)\tau$, the space $\mathcal{U}_0$ will be the piecewise constant space on this grid with support on the spatial domain $[0, (n-1/2)\tau]$ = $[0, T]$. We assume $L > T$. 
To construct our internal solutions, we find the mass matrix 
 \begin{equation} \label{massmatrixj1d} M_{kl}= \int_\Omega u_k u_l  dx \end{equation}
for $k,l =0,\ldots,n-1$ from the data by taking
\begin{equation} \label{massfromdataj1d} M_{kl} = \frac{1}{2} \left( F((k-l)\tau) + F((k+l)\tau) \right),\end{equation}  
and the background mass matrix $$ M_0 = \int_\Omega U_0^\top U_0, $$
where $U_0= [u^0_0, \ldots, u^0_{n-1} ] $ is the vector of background solutions. 
We compute the data generated snapshots ${\bf U}$ from the Cholesky decompositions
$$ M = LL^\top \ \ \ \ \ \ M_0 = L_0 L_0^\top, $$
and we set
\begin{equation} \label{datageneratedfields1d}  {\bf U}  := U_0 (L_0^\top)^{-1} L^\top = U_0 T \end{equation}
where $T$ is upper triangular.
The following estimate is well known, see for example \cite{BrSc}.
  \begin{proposition}\label{femapprox} For any $v\in H^1([0,T])$  for $T=(n-1/2)\tau$, let $P_\tau v $ be its $L^2$ projection onto the piecewise constant space $\mathcal{U}_0= \mbox{span}\{u^0_0, \ldots u^0_{n-1}\}$. Then there exists $C$ independent of $v$ and $\tau$ such that 
 $$ \| v- P_\tau v\|_{L^2([0,T])} \leq C\tau\| v\|_{H^1([0,T])}.$$
 \end{proposition}
We also will need a convenient version of a standard estimate. 
\begin{lemma}\label{waveest} Let $U= [u_0, \ldots, u_{n-1} ] $ be the vector of true internal snapshots, and let $U_0=[u^0_0,\ldots,u^0_{n-1} ]$ be the corresponding background snapshots. Then for $\tau = T/(n-1)$ small enough, there exists $C_T$ depending on $T$ but independent of $U$, $q$ and $n$ such that 
\begin{equation} \| u_i-u^0_i\|_{H^1(\Omega)} \leq C_T\| q\|_\infty { \| U \|_2\over{\sqrt{n}}}.    \end{equation}
for each $i=0,\ldots,n-1$.
\end{lemma}
\begin{proof}
    If $u$ solves (\ref{waveeq})-(\ref{waveeqid}) and the background solution $u_0$ solves (\ref{waveeq})-(\ref{waveeqid}) with $q(x)=0$, then the difference $u-u_0$ satisfies
\begin{equation}\label{wavediff}
(u-u_0)_{tt}  -\Delta (u-u_0)= - q(x) u \ \ \mbox{in}  \ \Omega\times [ 0,\infty)
\end{equation}
with initial conditions
\begin{eqnarray} \label{wavediffic}
(u-u_0) ( t=0) &=& 0 \ \mbox{in}\  \Omega  \\
(u-u_0)_t ( t=0) &=&  0\  \mbox{in}\  \Omega\\ {\partial\over{\partial\nu}}(u-u_0) &=& 0 \ \mbox{on} \ \partial\Omega. \end{eqnarray}
From standard energy estimates we have (see for example Theorem of \cite{Ev}) that 
\begin{equation} \sup_{t\in [0,T]} \| u-u_0\|_{H^1(\Omega)} \leq C_T \int_0^T \| q(x) u\|_{L^2(\Omega)}dt,  \end{equation}
which implies that 
\begin{equation} \sup_{t\in [0,T]} \| u-u_0\|_{H^1(\Omega)} \leq C_T\| q\|_\infty \int_0^T \| u\|_{L^2(\Omega)}dt.  \end{equation}
Hence we have for each $i=1,\ldots, n$,
\begin{equation} \| u_i-u^0_i\|_{H^1(\Omega)} \leq C_T\| q\|_\infty \int_0^T \| u\|_{L^2(\Omega)}dt.   \end{equation}
Consider the integral in the right hand side above. Since $\| u\|_{L^2(\Omega)}$ is continuous in time, its Reimann sums, including a left hand rule, will converge. So, by adjusting with a constant, for example $2$, we have for $\tau =T/(n-1)$ small enough 
\begin{eqnarray} \int_0^T \| u\|_{L^2(\Omega)}dt &\leq &  2\tau \sum_{i=0}^{n-1} \| u_i\|_{L^2(\Omega)} \nonumber \\
&\leq& 2 \tau \sqrt{n} ( \sum_{i=0}^{n-1} \| u_i \|^2_{L^2(\Omega)} )^{1/2} \\  &\leq & {2T\over{\sqrt{n}}} \| U \|_2. \end{eqnarray}

\end{proof}
\begin{remark} \label{L1remark} We note that we also also have that 
\begin{equation} \label{linfinityest} \sup_t \| u-u_0\|_\infty \leq C_T\end{equation}
where $C_T$ is independent of $h$. Using the one dimensional Green's function for the wave equation with source $qu$, we get that $u-u_0$ can be written as an integral of $qu$, yielding
\begin{equation} \sup_t \| u-u_0\|_\infty \leq C_T \| q\|_\infty \sup_t \| u(\cdot,t)\|_{L^1(\Omega)}. \end{equation} It is well known that $$ \sup_t \| u(\cdot,t)\|_{L^1(\Omega)}\leq C_T \| u_0(\cdot,t)\|_{L^1(\Omega)}, $$ which is bounded.
\end{remark}
 We note now that since the background snapshots are in the approximation space $\mathcal{U}_0$, we have that 
 $$ u_i -P_\tau{u}_i = (u_i-u^0_i) - P_\tau (u_i-u^0_i),$$
 so we can apply Proposition \ref{femapprox} to $v=u_i-u^0_i$.
 Combining this with Lemma \ref{waveest}, we get 
 \begin{lemma}\label{approxlem} Let $U= [u_0, \ldots, u_{n-1} ] $ be the vector of true internal snapshots, and let $ P_\tau U = [P_\tau{u}_0, \ldots, P_\tau{u}_{n-1} ]$ be their $L^2$ projections onto $\mathcal{U}_0$. Then for $\tau = T/(n-1/2)$ small enough, there exists $C_T$ depending on $T$ but independent of $U$, $q$ and $n$ such that 
\begin{equation} \| u_i-P_\tau{u}_i \|_2 \leq C_T \tau \| q\|_\infty {\| U \|_2  \over{\sqrt{n}}}   \end{equation}
for $i=1,\dots,n$, and hence
\begin{equation} \| U-P_\tau{U} \|_2 \leq C_T \tau \| q\|_\infty  \| U \|_2.    \end{equation}
 \end{lemma}
 Note that for this example, the background snapshots are already orthogonal, and so $M_0$ and $L_0$ are diagonal matrices with condition number clearly bounded with respect to $\tau$. Furthermore, the sequential admissible projections,  $\hat{U}= [ \hat{u}_0,\ldots,\hat{u}_{n-1}]$, where each $u_i$  is projected onto $\mbox{span}\{ u^0_0, \ldots ,u^0_i\}$,
 are exactly the same as the orthogonal $L^2$ projections. That is,  $$\hat{U} = P_\tau U.$$ The diagonal elements of  $L_0$ are given by 
$$  (L_0)_{ii}= \| u_i^0\|_2 $$  and the diagonal elements of $\hat{L}$ are given by $$ \hat{L}_{ii} = < u_i, {u_i^0\over{\|u_i^0\|_2}}>.$$ 
 This next Lemma relates the diagonal elements (eigenvalues) of $\hat{L}$ to those of $L_0$.
 \begin{lemma} \label{alphabounds}  For $\tau$ small enough,
$$ | {\hat{L}_{ii} \over{ (L_0)_{ii}}} - 1| \leq C\sqrt{\tau} $$
for some $C$ independent of $n$, $\tau$ and $i$. 
\end{lemma}
\begin{proof}
We write  $$|{\hat{L}_{ii}\over{(L_0)_{ii}}} -1 | = | {< u_i -u^0_i, u^0_i> \over{\| u^0_i\|_2^2}}| \leq {\| u_i -u^0_i\|_2\over{\| u_i^0\|_2}} $$  from Cauchy-Schwartz. The result follows from (\ref{linfinityest})  since $\| u^0_i\|_2 = O(1/\sqrt{\tau})$.
\end{proof}
\begin{lemma}\label{boundedmass} The mass matrix $M$ given by (\ref{massmatrixj1d}) satisfies 
$$\kappa(M)= \| M\|_2\| M^{-1}\|_2 \leq C$$
for some $C$ independent of $\tau$  (and $n$). 
\end{lemma}
\begin{proof} From Lemma \ref{alphabounds}, since the eigenvalues of $\hat{L}$ are its diagonal, we clearly have $\kappa(\hat{L})$ bounded and therefore $\kappa(\hat{M})$ bounded with respect to $\tau$. Since by Lemma \ref{masslem} $$\| {M\over{n}}- {\hat{M}\over{n}}\|_2\leq{\| U-\hat{U}\|_2^2\over{n}}\leq C\tau^2\|U\|_2^2\leq C\tau $$
where $C$ is independent of $\tau$, $M/n$ is arbitrarily close to $\hat{M}/n$ in norm. Additionally $n\hat{M}^{-1}$ is bounded since $\hat{L}_{ii}$ is order $\sqrt{n}$ by Lemma \ref{alphabounds}. So ${M}$ must also have bounded condition number. 
\end{proof}
 Finally, by using Proposition \ref{mainprop}, we have the following estimate. 
  \begin{proposition} \label{convergence1d} Let $U= [u_0, \ldots, u_{n-1} ] $ be the vector of true internal snapshots, $M$ be the true mass matrix (\ref{massmatrixj}), let  ${\bf U}$  be the row vector of data generated internal fields given by (\ref{datageneratedfields1d}). 
Then for $\tau$ small enough, there exists $C_T$ depending on $T$ but independent of $U$, $q$ and $\tau$ such that 
\begin{equation} { \| {\bf U}- U \|_2\over{\sqrt{n}} } \leq  C_T\tau^2\| q\|^2_\infty\| U\|_2 + C_T{\tau}\| q\|_\infty{\| U\|_2\over{\sqrt{n}}}. \end{equation} 
\end{proposition}
\begin{proof}  Since the sequential projections $\hat{U}$ are truly orthogonal, $\hat{U}=P_\tau U$, and the $R$ in Lemma \ref{masslem} is zero.  Furthermore, from Lemma \ref{alphabounds}, we know that the diagonal entries of $\hat{L}$ must be positive for $\tau$ small enough, since the diagonals of $L_0$ are clearly positive. Hence the diagonal entries of $\hat{T}=L_0^{-\top}\hat{L}^\top$ must be positive. We now can apply Proposition \ref{mainprop} and Lemma \ref{approxlem}, which yields
\begin{equation} { \| {\bf U}- U \|_2 \over{\sqrt{n}}} \leq  C_T \tau^2\kappa(M)\| q\|^2_\infty\| U\|_2+ C_T {\tau}\| q\|_\infty{\| U \|_2\over{\sqrt{n}}}.  \end{equation} 
The result follows from Lemma \ref{boundedmass}.
\end{proof}
\begin{remark} We note that Proposition \ref{convergence1d} shows convergence of order $\sqrt{\tau}$ since $\| U \|_2$ is $O(\sqrt{n}/\sqrt{\tau})$. 
\end{remark}

\section{An example with a piecewise linear initial wave}\label{sec:pwlinear}
Let us again consider the one dimensional case of (\ref{waveeq}-\ref{waveeqid}) where $\Omega = (0,L)\subset \mathbb{R} $ for $L$ large, 
\begin{equation}\label{waveeq1d2}
u_{tt}  - u_{xx} + q(x) u=   0 \ \ \mbox{in} \  (0,L) \times [ 0,\infty)
\end{equation}
with initial and boundary conditions
\begin{eqnarray} \label{waveeqic1d2}
u ( t=0) &=& g \ \mbox{in}\ (0,L)  \\
u_t ( t=0) &=&  0\  \mbox{in}\  (0,L) \\  u_x(x=0) =u_x(x=L) &=& 0 \ \mbox{for all } \  t\in [0,\infty).\label{waveeqid1d2}\end{eqnarray}
In this section, we assume that $q$ is smooth enough so that for any $t\leq T<L$, $u - u^0 $ is in $H^2([0,L])$. We assume that $L>T$.
We again measure the solution back at the source at the $2n-1$ evenly spaced time steps $t= k\tau$ for $k=0,\dots, 2n-2$, 
\begin{equation}\label{datadef1d2} F(k\tau ) = \int_0^1 g(x) u(x,k\tau) dx .\end{equation}
In this case we will take initial data $g$ to be a continuous piecewise linear hat function source near $x=0$. Define the reference hat function 
\begin{equation} \phi(x) = \left\{ \begin{array}{cc} 1+x & -1\leq x \leq 0 \\  1- x  & 0\leq x \leq 1 \\  0 & \mbox{otherwise} \end{array} \right. , \end{equation}
 and for a given time sampling size $\tau$ let us take initial data \begin{equation} g(x) =2{\phi(x/\tau)\over{\tau}}.\label{initaldata2} \end{equation}
Again the scaling is chosen so that $g$ is an approximate $\delta$ pulse function, and for any $\tau$, $$\int_0^1 g(x)dx =1.$$ The background solution at the $k$th step for $k\geq 1$ will be  $$ u^0_k = u^0(x,k\tau ) = g(x- k\tau)/2 = {\phi((x-k\tau)/\tau)\over{\tau}},$$
and so the background snapshots are simply the standard piecewise linear finite element basis on a spatial grid with of step $\tau$. Defining $T= n\tau$, the space $\mathcal{U}_0$ will be the finite element space with support on the spatial domain $[0, n\tau]$ = $[0, T]$, zero at the right endpoint.  

To construct our internal solutions, we find the mass matrix 
 \begin{equation} \label{massmatrixj1d2} M_{kl}= \int_\Omega u_k u_l  dx \end{equation}
for $k,l =0,\ldots,n-1$ from the data by taking
\begin{equation} \label{massfromdataj1d2} M_{kl} = \frac{1}{2} \left( F((k-l)\tau) + F((k+l)\tau) \right),\end{equation}  
and the background mass matrix $$ M_0 = \int_\Omega U_0^\top U_0, $$
where $U_0= [u^0_0, \ldots, u^0_{n-1} ] $ is the vector of background solutions. 
We compute the data generated snapshots ${\bf U}$ from the Cholesky decompositions
$$ M = LL^\top \ \ \ \ \ \ M_0 = L_0 L_0^\top, $$
and we set
\begin{equation} \label{datageneratedfields1d2}  {\bf U}  := U_0 (L_0^\top)^{-1} L^\top = U_0 T \end{equation}
where $T$ is upper triangular. 
The following estimate is well known, see for example \cite{BrSc}.
  \begin{proposition}\label{femapprox2} For any $v\in H^1([0,T])$ with $v(x)=0$ for $x\geq (i+1)\tau$, $i\leq n-1$, let $\hat{v}$ be its $L^2$ projection onto the piecewise linear space $\mathcal{U}_0= \mbox{span}\{u^0_0, \ldots u^0_{i}\}$. Then there exists $C$ independent of $v$ and $\tau$ such that 
 $$ \| v- \hat{v}\|_{L^2([0,T])} \leq C\tau\| v\|_{H^1([0,T])}.$$
 \end{proposition}
Since the background snapshots are in the approximation space $\mathcal{U}_0$, we have that 
 $$ u_i -\hat{u}_i = (u_i-u^0_i) - \widehat{(u_i-u^0_i)},$$
 so we can apply Proposition \ref{femapprox2} to $v=u_i-u^0_i$, which due to causality is zero for $x\geq (i+1)\tau$.
 \begin{lemma}\label{approxlem2} Let $U= [u_0, \ldots, u_{n-1} ] $ be the vector of true internal snapshots, and let $\hat{U} = [\hat{u}_0, \ldots, \hat{u}_{n-1} ]$ be the sequential causal $L^2$ projections described above. Then for $\tau = T/n$ small enough, there exists $C_T$ depending on $T$ but independent of $U$, $q$ and $n$ such that 
\begin{equation} \| u_i-\hat{u}_i \|_2 \leq C_T \tau  \| u_i-u^0_i \|_{H^1([0,T])}    \end{equation}
for $i=1,\dots,n$, and hence
\begin{equation} \| U-\hat{U} \|_2 \leq C_T \tau \| U-U_0 \|_{[H^1([0,T])]^n}.    \end{equation}
 \end{lemma}
 
We again will need a Lemma which relates the diagonal elements (eigenvalues) of $\hat{L}$ to those of $L_0$.
 \begin{lemma} \label{alphabounds2} Define $\bar{u}_i^0$ to be the sequentially orthonormalized background snapshots, 
then $$ \hat{L}_{ii} = < u_i, \bar{u}_i^0>, $$
and 
$$ (L_0)_{ii}= < u_i^0, \bar{u}_i^0>  $$ for $u_i,u_i^0$ the true and background snapshots, satisfy
$$ | {\hat{L}_{ii} \over{ (L_0)_{ii}}} - 1| \leq C\sqrt{\tau} $$
for some $C$ independent of $n$, $\tau$ and $i$. 
\end{lemma}
\begin{proof}
The proof is similar to Lemma \ref{alphabounds} since again $u_i - u_i^0$ is small compared to $u_i^0$.
\end{proof}
We again have that $\kappa(M)$ is bounded independent of $\tau$.
\begin{lemma}\label{boundedmass2} The mass matrix $M$ given by (\ref{massmatrixj1d}) satisfies 
$$\kappa(M)= \| M\|_2\| M^{-1}\|_2 \leq C$$
for some $C$ independent of $\tau$  (and $n$). 
\end{lemma}
\begin{proof} We first note that the background mass matrix $M_0$ is simply the mass matrix for the standard one dimensional piecewise linear basis functions on $[0,T]$ on a uniform mesh of size $\tau$, multiplied by $1/\tau^2$.  This scaling does not affect the condition number, and it is well known that $\kappa(M_0)\leq 3$.  We remark that $6\tau M_0 $ is the tridiagonal matrix with $4$ on the diagonal and $1$ on the off diagonal, so this is the scaling in which the eigenvalues are $O(1)$. This means that the condition number of $L_0$ is bounded. Due to Lemma \ref{alphabounds2}, the rest of the argument follows the same as in Lemma \ref{boundedmass}. 
\end{proof}
 Unlike the example from the previous section, the background snapshots are overlapping, and the admissible projections $\hat{U}$ are not the same as the full orthogonal projections $P_\tau U$. The full projection onto $\mathcal{U}_0$ will in general have a contribution from $u^0_{i+1}$. Indeed,
 $$ P_\tau u_i = \hat{u}_i + < u_i, \bar{u}^0_{i+1}> \bar{u}^0_{i+1}$$
 where $\{ \bar{u}_i^0\}$ are the orthonormalized background snapshots. This means that $R$ in Lemma \ref{masslem} and Proposition \ref{mainprop} will be nonzero on its super and sub diagonal. 
  \begin{proposition} \label{convergence1d2} Let $U= [u_0, \ldots, u_{n-1} ] $ be the vector of true internal snapshots, $M$ be the true mass matrix (\ref{massmatrixj}), let  ${\bf U}$  be the row vector of data generated internal fields given by (\ref{datageneratedfields1d}). 
Then for $\tau$ small enough, there exists $C$ depending on $T$ and $q$ but independent of $U$ and $\tau$ such that 
\begin{equation} {\| {\bf U}- U \|_2\over{\sqrt{n}} } \leq  C\sqrt{\tau}\left(1 + \sup_i \| u_i-\hat{u}_i\|_\infty\right). \end{equation} 
\end{proposition}
\begin{proof} We first show that $\hat{U}$, the vector of $L^2$ projections of $U$ onto the partial subspaces, is admissible. For $\hat{U} = U_0 \hat{T}$, it is clear that $\hat{T}$ is upper triangular by the definition of $\hat{U}$. Furthermore, from Lemma \ref{alphabounds2}, we know that the diagonal entries of $\hat{L}$ must be positive for $\tau$ small enough, since the diagonals of $L_0$ are clearly positive. Hence the diagonal entries of $\hat{T}=L_0^{-\top}\hat{L}^\top$ must be positive, and we can apply Proposition \ref{mainprop} and Lemma \ref{approxlem2}. As explained above, $R$ is not zero in this case, but is nonzero only on the two off diagonals. Holder's inequality gives 
$$| R_{i,i+1}|\leq \| u_i-\hat{u}_i\|_\infty\|\hat{u}_{i+1}\|_1\leq C\| u_i-\hat{u}_i\|_\infty\| {u}_{i+1}\|_1, $$
since $u_{i+1}$ and its projection are close in $L^2$, and are therefore close in $L^1$ on a bounded domain. Furthermore, each $u_{i+1}$ is bounded in $L^1$ since the initial waves are, see Remark \ref{L1remark}. Hence by possibly adjusting the constant we have $$| R_{i,i+1}|\leq C\| u_i-\hat{u}_i\|_\infty. $$ The estimate is done similarly for the subdiagonal. From the definition of the Frobenius norm and the sparsity of $R$, this yields
$$\| R\|_F\leq C \sqrt{n} \sup_i \| u_i-\hat{u}_i\|_\infty.$$ 
Using Proposition \ref{mainprop}, 
\begin{equation} { \| {\bf U}- U \|_2\over{\sqrt{n}} } \leq  \kappa(M)\left( {\| U-\hat{U}\|^2_2\over{\|U\|_2}} + C\sqrt{n} {\sup_i\| u_i-\hat{u}_i\|_\infty\over{\| U\|_2}}\right) +{\| U-\hat{U}\|_2\over{\sqrt{n}}}.  \end{equation} 
We know from direct calculation that $$  {\| U_0 \|_2\over{\sqrt{n}}} = {\sqrt{2/3}\over{\sqrt{\tau}}},  $$ so from the boundedness of ${\| U- U_0 \|_2\over{\sqrt{n}}} $ it follows that there exists $c>0$ such that 
$$  {\| U \|_2\over{\sqrt{n}}} \geq  c {1\over{\sqrt{\tau}}}.  $$ Hence we can adjust the constant $C$ so that 
\begin{equation} { \| {\bf U}- U \|_2\over{\sqrt{n}} } \leq  \kappa(M)\left( {\| U-\hat{U}\|^2_2\over{\|U\|_2}} + C\sqrt{\tau} \sup_i\| u_i-\hat{u}_i\|_\infty\right) +{\| U-\hat{U}\|_2\over{\sqrt{n}}}.  \end{equation} 
By Lemma \ref{approxlem2} and Lemma \ref{waveest}, 
\begin{equation} { \| {U}- \hat{U} \|_2\over{\sqrt{n}} } \leq  C {\tau}\kappa(M) {\| U\|_{2}\over{\sqrt{n}}}. \end{equation} 
and so the result follows from Lemma \ref{boundedmass2} and again adjusting the constant $C$.  \end{proof}

\section{Higher Dimensions} \label{sec:mimo}
The results in the Section \ref{sec:generalbounds} are for the SISO (single input/single output) setup, and hold in any dimension. For small enough $q$, the error in the data generated snapshots will be controlled by the error in the best approximation from the background snapshot space. In one spatial dimension, this is enough to obtain accurate approximations. In higher dimensions, snapshots coming from one source will not be rich enough to obtain a good approximation of the field, even if we were to have the best approximation exactly. (In \cite{DrMoZa3,DrMoZa4} it was shown numerically that for SISO data spherical averages of the fields can be reconstructed.)  In order to obtain accurate data generated solutions in higher dimensions, one needs to have more sources and receivers.  This will require a MIMO (multiple input/multiple output) setup with several source/receiver pairs. This is modeled by 
\begin{eqnarray} \label{wavemodel2}
u_{tt}  + (-\Delta +q)u&= &  0 \ \mbox{in}  \ \Omega\times [ 0,\infty) \\
u ( t=0) &=& g_j \ \mbox{in}\  \Omega  \\
u_t ( t=0) &=&  0\  \mbox{in}\  \Omega \end{eqnarray}
where $\Omega $ is a domain in $\mathbb{R}^d$ and each  $g_j$, for $j=1,\ldots K$, is a pulse localized near a point $x_j$. Denote by $u^{(i)}(x,t)$ the field corresponding to source $g_i$, and assume that we have the full response matrix
\begin{eqnarray}\label{datadef2} F^{ji} (k\tau) &=&  \int_\Omega g_j(x) u^{(i)}(x,t) dx \\ &=& \int_\Omega g_j(x) \cos{(\sqrt{-\Delta+q}k\tau)} g_i(x) dx ,\end{eqnarray}
for $j,i=1,\dots, K $, and $k=0,\ldots, 2n-1$. The component $F^{ji}(k\tau)$ represents the response at receiver $j$ from source $i$ at time $k\tau$.  Define $$u^{(i)}_k= u^{(i)}(x,k\tau)$$ to be the $k$th snapshot of $u^{(i)}$. 
For this problem, the mass tensor is given by 
\begin{equation} \label{massmatrixMIMO} M_{klij}= \int_\Omega u^{(i)}_k u^{(j)}_l  dx,  \end{equation}
for $k,l =0,\ldots, n-1$ and $i,j=1,\ldots,K$. The tensor $M$ can be viewed as an $n\times n$ matrix of blocks of size $K\times K$. The  $K\times K$ block $M_{kl}$ can be obtained directly from the data by the extension of (\ref{massfromdataj}) to blocks: 
\begin{equation} \label{massfromdataMIMO} M_{kl} = \frac{1}{2} \left( F((k-l)\tau) + F((k+l)\tau) \right),\end{equation} see for example \cite{BoGaMaZi}. Let $$U=[u^{(1)}_0, u^{(2)}_0, \ldots, u^{(K)}_0, u_1^{(1)}, u^{(2)}_1, \ldots, u^{(K)}_1, \ldots, u^{(1)}_{n-1}, u^{(2)}_{n-1},\ldots, u_{n-1}^{(K)}],$$ be a row vector of the true internal solutions, ordered first by time step and second by some fixed ordering of the sources. Similarly, let $$U_0=[(u_0^0)^{(1)}, \ldots, (u_0^0)^{(K)}, \ldots, (u^0_{n-1})^{(1)},\ldots, (u^0_{n-1})^{(K)}],$$ be the corresponding background fields, ordered in the same way, and let $M_0$ be the corresponding background mass matrix. To compute the data generated internal fields in this case,  we compute a block Cholesky decomposition, 
$$ M = L L^\top$$ 
where $L$ is block lower triangular.  Recall that in the SISO case the diagonal of $L$ is chosen to be positive, giving us uniqueness. In the block Cholesky case there is more non-uniqueness since there is a choice of the square root of the positive definite diagonal blocks. For the SISO problem, we saw that due to the asymptotics, this non-uniqueness in the  $\pm$ of the square root used in the Cholesky decomposition does not affect the final representation (which has diagonal asymptotically $= 1$).  In the MIMO case we believe that the matrix square root non-uniqueness will also not affect the final result, as the diagonals should be asymptotically the identity.

The choice of the square root of the diagonal blocks in the block Cholesky decomposition will determine the admissible space.  Once we have made such a choice, we perform a similar decomposition / orthogonalization for the background mass matrix
$$ M_0 =L_0 L_0^\top$$
and define the internal solutions generated directly from the data
$${\bf{U}} = {U}^0 (L_0)^{-\top} L^\top .$$
If $L$ are and $L_0$ are chosen to be lower triangular, the results of Section \ref{sec:generalbounds} carry over exactly. However, this choice is not physical since it would enforce spatial causality. Additionally, in practice conditioning issues can occur when the background solutions overlap. More analysis is required, and this is the subject of future work. 
\section{A Numerical Example}
We consider equation (\ref{waveeq}) in $\Omega=[0;200]$ with initial conditions and boundary conditions (\ref{waveeqic})-(\ref{waveeqid}) for piecewise linear $g(x)$ given by (\ref{initaldata2}) and $$q(x)=0.3e^{-0.04(x-70)^2}.$$ Though the initial condition is continuous, its derivative is discontinuous, so the numerical solution may exhibit the Gibbs effect. To avoid its propagation into the convergence study of our algorithm, we discretized (\ref{waveeq}) on a fine enough grid using $N=153600$ spatial grid steps. Then, for the corresponding spatial grid step $h=200/N$ we applied a finite-difference second-order scheme with time-step $\Delta t=h/2$ to generate the data on the time interval $[0;T]$ for $T=100$. We considered data sampling with $\tau=T/n$ for $n=75,~150,~300,~600$. In Fig.\ref{fig:sol} we plotted (for each of the sampling rates) for the terminal time $T$, the reconstructed solution $\bf{U}$, the true solution $U$, background solution $U_0$, and the causal projection $\hat{U}$ of $U$ onto the subspace of background snapshots. The reconstructed solutions match the true one pretty well, even for $n=75$. Then, in Fig. \ref{fig:err} we plotted the errors of reconstructions for various values of $n$ for the terminal time $T$. As one would expect, the error is mostly concentrated near the singular part of the solution, Finally, in Fig. \ref{fig:conv} we plotted the convergence rate in the $L_2$ norm.  As one can observe, it matches well with the $\sqrt{\tau}$ error decay.
\begin{figure}
    \centering
    \begin{tabular}{cc}
    \includegraphics[width=.45\textwidth]{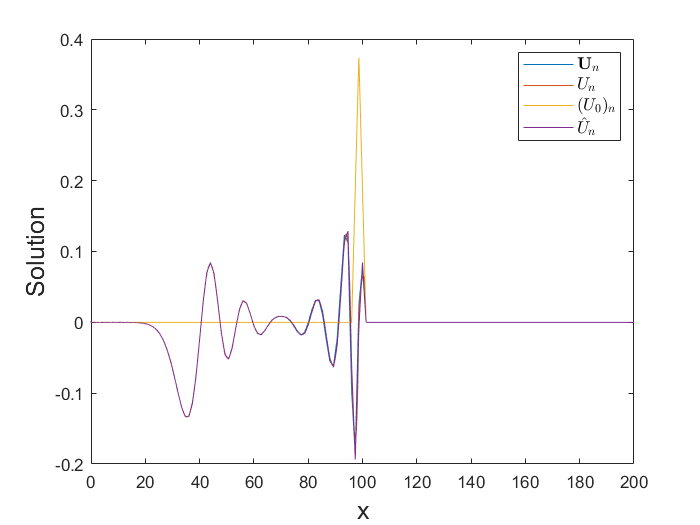} &
    \includegraphics[width=.45\textwidth]{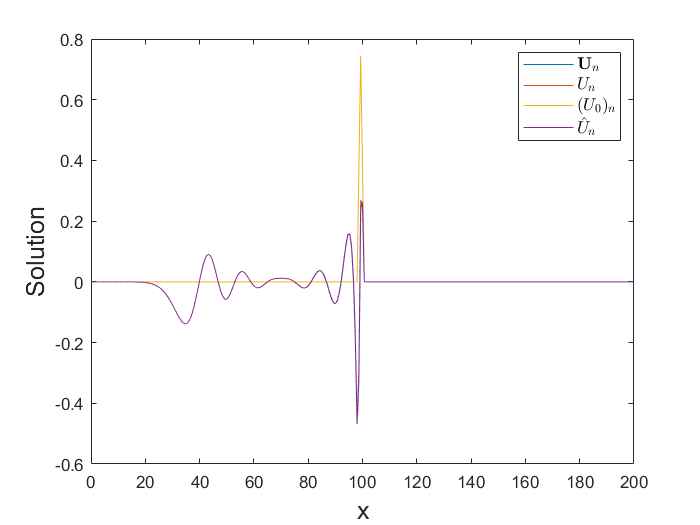} \\
    a) $n=75$ & b) $n=150$ \\
    \includegraphics[width=.45\textwidth]{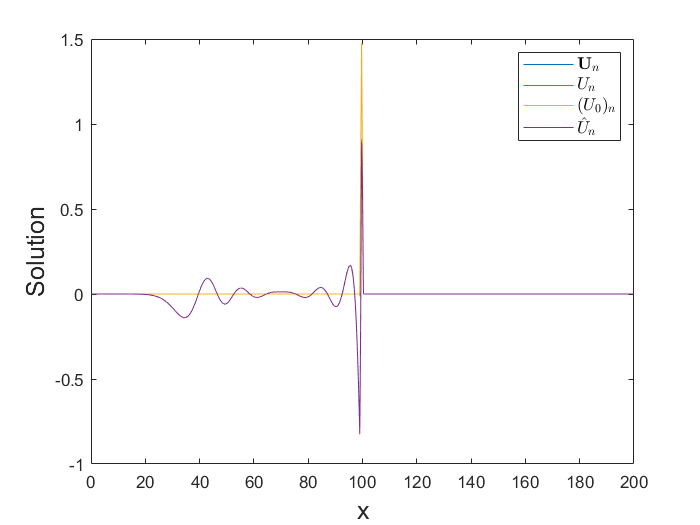} &
    \includegraphics[width=.45\textwidth]{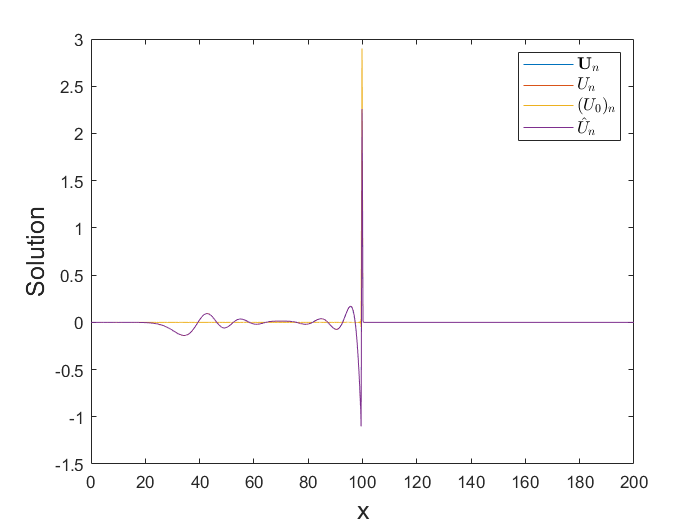} \\
    c) $n=300$ & d) $n=600$ \\
    \end{tabular}    
    \caption{Reconstructed solution ${\bf{U}}$, true solution $U$, background solution $U_0$ and causal projection $\hat{U}$ of $U$ onto the the subspace of background snapshots for the terminal time $T=100$ and the  samplings $n=75$ (a), $n=150$ (b), $n=300$ (c), $n=600$ (d).
    }
    \label{fig:sol}
\end{figure}
\begin{figure}
    \centering
    \includegraphics[width=\textwidth]{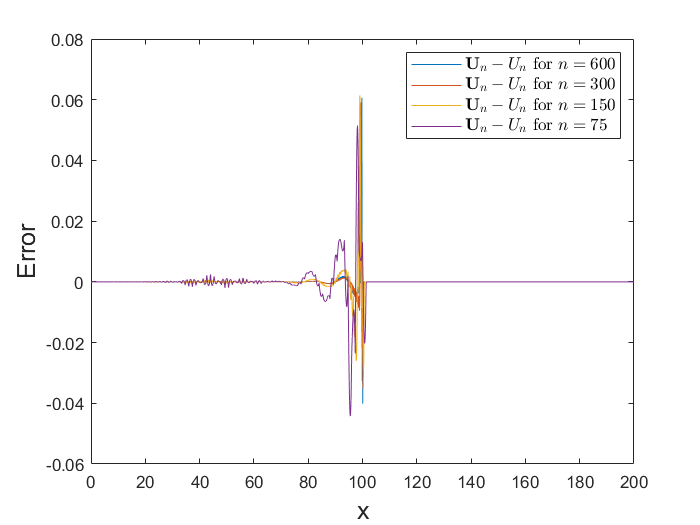}
    \caption{Errors of the reconstructed solutions for the terminal time $T=100$. As one can expect, the error is mostly concentrated near the singularity.
    }
    \label{fig:err}
\end{figure}
\begin{figure}
    \centering
    \includegraphics[width=\textwidth]{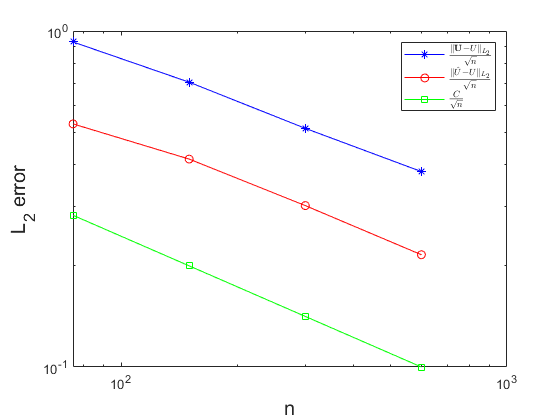}
    \caption{Convergence rates of the approximants ${\bf U}$ and causal projection $\hat{U}$. As one can observe, both are aligned pretty well with the $O(\sqrt{\tau})$ error estimate.
    }
    \label{fig:conv}
\end{figure}

\section{Discussion}
First, we should note that the more commonly used variable speed wave equation  $$\frac{1}{c^2(z)} w_{tt}-w_{zz}=0$$ can be transformed into the plasma wave equation via the application of consecutive travel-time and Liouville transforms, so our analysis can be extended to this problem as well. Extensions to variable wave-speed multidimensional formulations, however, would be more difficult due to the variable metric in geometric optics. Also, our analysis can be extended to the frequency domain problems in the Loewner framework  \cite{borcea2020reduced} by using the Lanczos-Cholesky analogy, and this is the subject of future work. 

As discussed in Section \ref{sec:mimo}, extensions of this work to the multidimensional MIMO formulations \cite{druskin2016direct} for the plasma wave equation are also possible using a block-Cholesky formulation. Due to the non-uniqueness of the square roots of the positive definite blocks and conditioning issues that occur when the background solutions begin to overlap, this will require more analysis. 

There is also a possibility of using our approach for the analysis of related approaches such as Marchenko redatuming \cite{Vargas2021ScatteringbasedFF} and boundary control
\cite{Belishev_2007,doi:10.1137/17M1151481}, but these 
are yet to be investigated.

\thanks{{\bf Acknowledgments.} The authors would like to thank Tarek Habashy for many helpful discussions. 
V. Druskin was partially supported by AFOSR grants FA 9550-20-1-0079, FA9550-23-1-0220,  and NSF grant  DMS-2110773. M. Zaslavsky was partially supported by  AFOSR grant  FA9550-23-1-0220. S. Moskow was partially supported by NSF grants DMS-2008441 and DMS-2308200.  }
\bibliographystyle{siamplain}
\bibliography{ROMbibliography}
\end{document}